\documentclass[a4paper]{article}
\usepackage[utf8]{inputenc}

\usepackage{amsmath,amsthm,amssymb}
\usepackage{geometry}
\usepackage{color}
\usepackage{cite}
\usepackage{authblk}

\usepackage{tikz}
\usepackage{pgfplots}
\usepackage{subdepth}
\usepackage{enumerate}
\usepackage{authblk}
\usepackage{mathdots}
\usepackage{mathtools}
\usepackage{graphicx}
\usepackage{siunitx}
\usepackage{hyperref}
\usepackage{xcolor}
\usepackage{subfig}
\usepackage{float}

\usepackage{algorithm}
\usepackage{algpseudocode}

\usepackage{verbatim}
\usepackage{hyperref}
\usepackage{amsmath,amsthm,amssymb}
\usepackage{color}
\newtheorem{theorem}{Theorem}[section]
\newtheorem{lemma}[theorem]{Lemma}
\newtheorem{remark}[theorem]{Remark}
\newtheorem{coro}[theorem]{Corollary}
\newtheorem{defi}[theorem]{Definition}

\newtheorem{example}[theorem]{Example}

\newcommand{\C}{\mathbb C}
\newcommand{\ii}{\mathfrak i}

\newcommand{\la}{\lambda}

\newcommand{\diag}{{\rm diag}}

\newcommand{\orbstar}{{\mathcal O^\star}}
\newcommand{\bun}{{\mathcal B}}
\newcommand{\bunc}{{\mathcal B^c}}
\newcommand{\bunh}{{\mathcal B^*}}
\newcommand{\bunstar}{{\mathcal B^\star}}
\newcommand{\cbunc}{{\overline{\mathcal B^c}}}
\newcommand{\cbunh}{{\overline{\mathcal B^*}}}
\newcommand{\cbunstar}{{\overline{\mathcal B^\star}}}

\title{The generic canonical form for $^\star$congruence of matrices}
\author[1]{Fernando De Ter\'an\thanks{\tt fteran@math.uc3m.es}}
\author[1] {Froil\'an M. Dopico\thanks{\tt dopico@math.uc3m.es}}
\affil[1]{Universidad Carlos III de Madrid, ROR: https://ror.org/03ths8210, Departamento de Matemáticas, Avda. Universidad 30, 28911, Legan\'es, Spain}
\date{}

\pgfplotsset{compat=1.18}
\begin{document}

\maketitle

\begin{abstract}
    First, we prove that the set of $n\times n$ complex matrices is the closure of a certain open subset whose elements have a very specific canonical form under congruence, which is uniquely determined up to the values of some parameters, but which has a slightly different expression depending on whether $n$ is even or odd. As a consequence, the canonical form under congruence of the elements of this subset can be considered the generic canonical form under congruence of complex $n\times n$ matrices. Second, we prove that the set of $n\times n$ complex matrices is the union of the closures of certain $\lfloor n/2\rfloor+1$ open subsets and that, for each of these subsets, its elements have a very specific canonical form under $^*$congruence, which is uniquely determined up to the values of some parameters. As a consequence, the $\lfloor n/2\rfloor+1$ canonical forms under $^*$congruence of the elements of each of these subsets can be considered  the generic canonical forms under $^*$congruence of complex $n\times n$ matrices. So, there is only one generic canonical form under congruence whereas the number of generic canonical forms under $^*$congruence is  $\lfloor n/2\rfloor+1$ instead, which reveals a strong dichotomy between the relations of congruence and $^*$congruence with respect to generic structures. In other words, we determine in this paper the generic matrix representations of $n\times n$ bilinear and sesquilinear forms in $\mathbb{C}^n \times \mathbb{C}^n$.
\end{abstract}


\bigskip


\section{Introduction}
Two matrices $A,B\in\C^{n\times n}$ are said to be {\em congruent} (respectively, {\em $^*$congruent}) if there is an invertible matrix $P\in\C^{n\times n}$ such that $P^\top AP=B$ (resp., $P^*AP=B$), where $M^\top$ (resp., $M^*$) denotes the transpose (resp., conjugate transpose) of $M$. For simplicity, we will use the notation $\star$ to denote both the transpose ($\top$) and the conjugate transpose ($*$). Both congruence and $^*$congruence are equivalence relations on $\C^{n\times n}$, and this allows us to classify the set of $n\times n$ complex matrices into the corresponding equivalence classes. Each equivalence class is the {\em orbit} under $^\star$congruence of a certain matrix $A\in\C^{n\times n}$, namely:
\begin{equation}\label{orbit}
    \orbstar(A):=\left\{P^\star AP:\ P\in\C^{n\times n}\ \mbox{invertible}\right\}.
\end{equation}
The search for a set of ``canonical" representatives of the equivalence classes has led to the ``canonical forms for congruence and $^*$congruence" (CFC and $^*$CFC, respectively) introduced in \cite[Th. 1.1]{hs06} (later reformulated by means of tridiagonal matrices in Theorems 1.1 and 1.2 in \cite{fhs08}). Following the notation above, we will use $^\star$CFC to denote both the CFC and the $^*$CFC.

The question that motivates the present work is the following: which is the most-likely $^\star$CFC? In order to answer this question we first need to clarify what we mean by ``most likely". The approach we follow here uses the notion of {\em $^\star$congruence bundle}, which is revised in Section \ref{basic_sec} and further studied in Section \ref{sec.techlemmas}. Such a bundle is a union of orbits with a common feature. To be more precise, the $^\star$CFC is a direct sum of blocks of three different types (see Theorem \ref{hs_th}). For $\star=*$, two of these three types of blocks depend on some complex values, namely:
\begin{itemize}
    \item unit complex numbers, $\alpha\in\C$, with $|\alpha|=1$ (Type I blocks), and
    \item complex numbers, $\mu\in\C$, with $|\mu|>1$ (Type II blocks),
\end{itemize}
and, for $\star=\top$ instead, one of the three types of blocks depends on:
\begin{itemize}
    \item complex numbers, $\mu\in\C$, with $0\neq \mu\neq(-1)^{k+1}$ (Type II blocks with size $2k\times 2k$).
\end{itemize}

A $^\star$congruence bundle is the set of matrices with the same $^\star$CFC up to the values of the complex numbers $\alpha$'s and $\mu$'s mentioned above (see Definitions \ref{bundlec_def} and \ref{bundle*_def}), which in addition must satisfy some constraints that guarantee that all the orbits in the same bundle have the same dimension. This definition mimics, though it is more complicated, the one of bundles under similarity of $n\times n$ complex matrices, namely, a bundle under similarity is the set of matrices with the same Jordan canonical form (JCF), up to the values of the eigenvalues, provided that the sets of distinct eigenvalues and the sizes of the corresponding Jordan blocks are the same (see \cite[\S 5.3]{arnold71}). It is well known that the ``most likely" JCF of complex $n\times n$ matrices is $\diag(\la_1,\hdots,\la_n)$, with $\la_i\neq\la_j$ for $i\neq j$, namely, most $n\times n$ complex matrices have $n$ distinct complex eigenvalues.  From a topological point of view, this means, first, that, for every matrix $A\in\C^{n\times n}$, there is a sequence of matrices $\{A_k\}_{k\in\mathbb N}$, with $A_k\in\C^{n\times n}$ for all $k\in\mathbb N$, which converges to $A$, and such that the JCF of $A_k$ is $\diag(\la_1^{(k)},\hdots,\la_n^{(k)})$, with $\la_i^{(k)}\neq\la_j^{(k)}$ for $i\neq j$, and, second, that the set of matrices whose JCF is of the type $\diag(\la_1,\hdots,\la_n)$, with $\la_i\neq\la_j$ for $i\neq j$, is an open subset of $\C^{n\times n}$. In terms of bundles, this means that the bundle under similarity of $\diag(\la_1,\hdots,\la_n)$, with $\la_i\neq\la_j$ for $i\neq j$, is {\em dense and open} in $\C^{n\times n}$, where the closure and the openness in this paper are considered in the standard topology of $\C^{n\times n}$. Accordingly, we say that $\diag(\la_1,\hdots,\la_n)$, with $\la_i\neq\la_j$ for $i\neq j$, is the {\em generic} JCF in $\C^{n\times n}$.

For $^\star$congruence we follow an approach like the one described in the previous paragraph for the similarity relation, though some important differences show up due to the existence of three types of different blocks in the $^\star$CFC, in contrast with the unique type of blocks existing in the JCF. Based on this approach, we will prove that $\C^{n\times n}$ is the closure of only one $^\top$congruence bundle ($^\top$congruence bundles are named just congruence bundles for simplicity), which is an open subset of $\C^{n\times n}$. This bundle is then termed the ``generic" congruence bundle and, accordingly, the CFC of the matrices within this bundle is the ``generic'' CFC of $n\times n$ complex matrices. This result is stated in parts (a) of the Theorems \ref{main_th} and \ref{th.opennness}. Thus, the situation for congruence is like the one for similarity in the sense that there is only one generic CFC. This scenario is in stark contrast with the one for $^*$congruence, since we prove in parts (b) of Theorems \ref{main_th} and \ref{th.opennness} that $\C^{n\times n}$ is the union of the closures of $\lfloor n/2\rfloor+1$ $^*$congruence bundles, and that each of these bundles is an open subset of $\C^{n\times n}$. These $\lfloor n/2\rfloor+1$  bundles are termed the ``generic" $^*$congruence bundles and the $^*$CFCs of the matrices within each of these bundles are the ``generic'' $^*$CFCs of $n\times n$ complex matrices.

In order to get more insight on the striking difference between the number of generic CFC and of generic $^*$CFCs, we explore the connection between an $n\times n$ complex matrix $A$ and the associated $\star$-palindromic pencil $A+\la A^\star$ \cite{MMMM-vibrations}. In particular, the number of generic $^\star$CFCs is related to the possible number of (distinct) $1 \times 1$ Type I blocks in the generic $^\star$CFCs of $n\times n$ complex matrices. These (distinct) Type I blocks correspond to different unit eigenvalues, i.e., eigenvalues of modulus $1$, of the pencil $A+\la A^\star$. When $\star=\top$ and $n$ is even, there are no, generically, unit eigenvalues in a $\top$-palindromic pencil $A+\la A^\top$, and when $n$ is odd there is just one (the eigenvalue $-1$, which is always an eigenvalue of a regular $\top$-palindromic pencil with odd size because $A-A^\top$ is skew-symmetric and $\det(A-A^\top )=0$). However, when $\star=*$, a generic $n\times n$ $*$-palindromic pencil can have $0,2,\hdots,n$ unit eigenvalues, which are different to each other (when $n$ is even), or $1,3,\hdots,n$ unit eigenvalues, which are different to each other (when $n$ is odd).

Another approach to obtain the generic canonical structures of matrices or matrix pencils, different from the topological one followed in this paper, is by means of the codimension of the orbits or bundles. More precisely, the generic canonical structures are those for which the corresponding bundles have codimension $0$ within the whole space of matrices or matrix pencils. This has been the approach followed in \cite[Cor. 7.1]{de95} for the {\em Kronecker Canonical Form} of rectangular matrix pencils, and in \cite{dd10} and \cite{dd11} for, respectively, the CFC and the $^*$CFC of $n\times n$ complex matrices. The result presented in Theorem \ref{main_th}-(a) for the CFC is in accordance with \cite[Th. 4]{dd10}, as expected (namely, the same unique generic CFC is obtained in both cases, which has slightly different expressions for $n$ even or odd). However, we want to highlight that Theorem 5.3 in \cite{dd11} is incomplete since it states that there is only one generic $^*$CFC, which, again, has slightly different expressions for $n$ even or odd. The generic $^*$CFC in \cite[Theorem 5.3]{dd11} is only one of the $\lfloor n/2\rfloor+1$ $^*$CFCs defined in Theorem \ref{main_th}-(b), more precisely that corresponding to $\ell = \lfloor n/2\rfloor$. The reason of this disagreement is not that the codimension approach fails in this case, but that \cite[Theorem 5.3]{dd11} did not identify correctly all the existing $^*$congruence bundles that have codimension $0$ (see Remark \ref{rem.codimension}).

In terms of bilinear and sesquilinear forms, Theorem \ref{main_th} provides the matrix representation of a generic bilinear or sesquilinear form. In particular, there is just one generic $n\times n$ bilinear form, whereas there are $\lfloor n/2\rfloor+1$ generic $n\times n$ sesquilinear forms. Further insights on the geometry and topology of the spaces of bilinear and sesquilinear forms, in terms of the changes of the canonical forms of the matrix representations, can be found in \cite{dfs12,dfs14,frs16}.

In the last section of the paper, we perform several numerical experiments with {\sc matlab} to confirm the genericity of the $^\star$CFC identified in Theorem \ref{main_th}. In these experiments, we have computed the eigenvalues of a large number of $n\times n$ $\top$-palindromic and $*$-palindromic pencils (for $n=24$ and $25$) generated with random matrices, and we have counted the number of unit eigenvalues of each pencil. Then, we have counted the number of pencils having exactly $m$ unit eigenvalues, for each $1\leq m\leq n$. Since, as we mentioned before, the distinct unit eigenvalues of a palindromic pencil correspond to distinct $1\times 1$ Type I blocks in the $^\star$CFC, we can infer the $^\star$CFC of the matrices that generate these pencils. The outcome of the experiments indicates that all generic $^\star$CFCs arise, and only these ones show up. This provides an experimental confirmation of Theorem \ref{main_th}.

\section{Canonical forms and bundles for congruence}\label{basic_sec}
By $I_k$ we denote the $k\times k$ identity matrix, and $\ii$ denotes the imaginary unit (that is, $\ii^2=-1$). $\C^{n\times n}$ denotes the set of $n\times n$ complex matrices. For any subset $\mathcal{W} \subseteq \C^{n\times n}$, we denote by $\overline{\mathcal{W}}$ its closure in the standard topology of $\C^{n\times n}$.

The CFC and the $^*$CFC were obtained in \cite{hs04} based on results in \cite{serg1988}. Direct proofs of these canonical forms were presented in \cite{hs06}. We will need the following matrices introduced in \cite{hs06}, for each $k\geq1$, and $\mu\in\C$, to describe these forms:
\begin{eqnarray}
   J_k(\mu):= \begin{bmatrix}
       \mu &1\\&\ddots&\ddots\\&&\mu&1\\&&&\mu
   \end{bmatrix}_{k\times k},\label{jordanblock}\\
   \Gamma_k:=\left[\begin{array}{c@{\mskip8mu}c@{\mskip8mu}c@{\mskip8mu}c@{\mskip8mu}c@{\mskip8mu}c}0&&&&&(-1)^{k+1}\\[-4pt]
&&&&\iddots&(-1)^k\\[-4pt]&&&-1&\iddots&\\&&1&1&&\\&-1&-1&&&\\1&1&&&&0\end{array}\right]_{k\times k},\label{gamma}\\
    H_{2k}(\mu):=\begin{bmatrix}
    0&I_k\\
     J_k(\mu)&0
    \end{bmatrix}_{2k\times2k}\label{h2k}.
\end{eqnarray}
The matrix in \eqref{jordanblock} is a Jordan block associated with the eigenvalue $\mu$ (see, for instance, \cite[Def. 3.1.1]{hj13}). We highlight that $\Gamma_1=1$ and $J_1(\mu)=\mu$.

The previous matrices are the building blocks of the CFC and $^*$CFC, which we recall in the next theorem as stated in \cite{hs06}.

\begin{theorem}{\rm ({Canonical form for congruence and for $^*$congruence}, \cite[Th. 1.1]{hs06})}.\label{hs_th}
   \begin{itemize}
\item[\rm(a)] Each complex square matrix is congruent to a direct sum, uniquely determined up to permutation of summands, of canonical matrices of the three types
\begin{center}
    {\renewcommand{\arraystretch}{1.6}
     \renewcommand{\tabcolsep}{0.3cm}
    \begin{tabular}{|c|c|}\hline
      {\rm  Type 0} & $J_k(0)$  \\\hline
      {\rm  Type I} & $\Gamma_k$ \\\hline
      {\rm  Type II} & \begin{tabular}{c}
           $H_{2k}(\mu)$, \;  $0\neq\mu\neq(-1)^{k+1}$, \\
           $\mu$ is determined up to replacement by $\mu^{-1}$
      \end{tabular} \\\hline
    \end{tabular}}
    \end{center}

   \item[\rm(b)] Each complex square matrix is $^*$congruent to a direct sum, uniquely determined up to permutation of summands, of canonical matrices of the three types
    \begin{center}
    {\renewcommand{\arraystretch}{1.6}
     \renewcommand{\tabcolsep}{0.3cm}
    \begin{tabular}{|c|c|}\hline
      {\rm  Type 0} & $J_k(0)$  \\\hline
      {\rm  Type I} & $\alpha \Gamma_k, \quad |\alpha|=1$\\\hline
      {\rm  Type II} &$H_{2k}(\mu), \quad |\mu|>1$\\\hline
    \end{tabular}}
    \end{center}
    \end{itemize}
\end{theorem}

\medskip

For the purposes of this work, it is convenient to express the CFC in Theorem \ref{hs_th}-(a) as in Corollary \ref{coro.CFC}. The objectives of this slight reformulation are to eliminate the indeterminacy of the possible replacement of $\mu$ by $\mu^{-1}$, which will simplify the statements and proofs of some results, and to distinguish the blocks $H_{2k}((-1)^k)$ from the blocks $H_{2k}(\mu)$ for other values of $\mu$. The reason is that the blocks $H_{2k}((-1)^k)$ behave very differently than $H_{2k}(\mu)$ for $\mu \ne (-1)^k$ with respect to the dimension of congruence orbits \cite[Th. 2]{dd10} and under perturbations \cite[Ths. 2.2, 2.3 and 3.5]{ddkks15}

\begin{coro} \label{coro.CFC}
Each complex square matrix is congruent to a direct sum, uniquely determined up to permutation of summands, of canonical matrices of the types
\begin{center}
    {\renewcommand{\arraystretch}{1.6}
     \renewcommand{\tabcolsep}{0.3cm}
    \begin{tabular}{|c|c|}\hline
      {\rm  Type 0} & $J_k(0)$  \\\hline
      {\rm  Type I} & $\Gamma_k$ \\\hline
      {\rm  Type II-(a)} &
           $H_{2k}(\mu)$, \;  $|\mu| >1$ or $\mu = e^{i \theta}$, with $0<\theta <\pi$
     \\\hline
     {\rm  Type II-(b)} &
           $H_{2k}((-1)^k)$
     \\\hline
    \end{tabular}}
    \end{center}
\end{coro}

\medskip

Definitions \ref{bundlec_def} and \ref{bundle*_def} introduce the notions of congruence and $^*$congruence bundles that are used in this work. They play a central role in the statement of the main results, Theorems \ref{main_th} and \ref{th.opennness}, because the generic CFC and $^*$CFC are defined as those that generate certain bundles which are open subsets of $\C^{n\times n}$ and such that the closure of their union is precisely $\C^{n\times n}$.

We emphasize that Definitions \ref{bundlec_def} and \ref{bundle*_def} do not coincide with the definitions of bundles in \cite[Def. 5.2]{dd11} and \cite[p. 62]{dd10}. The reason is that the definitions in \cite{dd10,dd11} are not fully satisfactory because, among other facts, not all congruence or $^*$congruence orbits included in those bundles have the same dimension, as can be checked with the dimension formulas in \cite[Th. 3.3]{dd11} and \cite[Th. 2]{dd10}. This drawback is fixed in Definitions \ref{bundlec_def} and \ref{bundle*_def}. As discussed in \cite[Sec. 6]{ddkks15} (see also the paragraph above Definition 1.2 in \cite{ddkks15}), it is not evident to give a definition of bundles of matrices under congruence which satisfy properties analogous to those of the bundles under similarity introduced in \cite{arnold71}. This led to the authors of \cite{ddkks15} to define that $A,B \in \C^{n\times n}$ belong to the same congruence bundle if and only if the pencils $A + \la A^\top$ and $B + \la B^\top$  belong to the same strict equivalence bundle, as defined in \cite[Sec. 3.1]{eek99}. This definition is based on the fact that $A,B \in \C^{n\times n}$ are congruent if and only if $A + \la A^\top$ and $B + \la B^\top$ are strictly equivalent \cite{roiter79}. However, it is not possible to extend this definition to $^*$congruence bundles because it is {\em not} true that $A,B \in \C^{n\times n}$ are $^*$congruent if and only if $A + \la A^*$ and $B + \la B^*$ are strictly equivalent. Therefore, we have decided to adopt Definitions \ref{bundlec_def} and \ref{bundle*_def}, which at least are enough for the purposes of this work.

\medskip

\begin{defi}\label{bundlec_def} {\rm (Congruence bundle).}
     Let the {\rm CFC} of $A\in\C^{n\times n}$ be $$\bigoplus_{j=1}^{n_1}J_{k_j}(0)\oplus\bigoplus_{j=1}^{n_2}\Gamma_{\ell_j}\oplus\bigoplus_{j=1}^{n_3} H_{2m_j}(\mu_j) \oplus \bigoplus_{j=1}^{n_4} H_{2 p_j} ( (-1)^{p_j}),$$ for some $\mu_j\in\C$ such that $|\mu_j| >1$ or $\mu_j = e^{\ii \theta_j}$, with $0< \theta_j < \pi$, for all $1\leq j\leq n_3$. Then, the {\em congruence bundle} of $A$ is the following subset of $\C^{n\times n}$
    \begin{align}\label{bundlec}
        \bunc(A) := & \left\{P \left(\bigoplus_{j=1}^{n_1}J_{k_j}(0)\oplus\bigoplus_{j=1}^{n_2}\Gamma_{\ell_j}\oplus\bigoplus_{j=1}^{n_3} H_{2m_j}(\widetilde \mu_j) \oplus \bigoplus_{j=1}^{n_4} H_{2 p_j} ( (-1)^{p_j}) \right)P^\top \right. \nonumber \\
        & \hspace*{0.8cm} :\ \left. \begin{array}{l}
             P\in \C^{n\times n} \; \mbox{\rm is invertible},  \\
            \mbox{$|\widetilde \mu_j| >1$ or $\widetilde \mu_j = e^{\ii \widetilde{\theta}_j}$, $0< \widetilde{\theta}_j < \pi$, for } 1\leq j \leq n_3,  \\[0.2cm]
              \widetilde\mu_i\neq\widetilde\mu_j \Longleftrightarrow \mu_i\neq\mu_j
        \end{array}\right\}.
    \end{align}
\end{defi}

\medskip

We note that the conditions imposed on the parameters $\widetilde{\alpha}_j$ in Definition \ref{bundle*_def} are related to the part 5 of the Theorem 3.3 in \cite{dd11}.
\begin{defi}\label{bundle*_def} {\rm ($^*$congruence bundle).}
    Let the {\rm$^*$CFC} of $A\in\C^{n\times n}$ be $$\bigoplus_{i=1}^{n_1}J_{k_i}(0)\oplus\bigoplus_{i=1}^{n_2}\alpha_i\Gamma_{\ell_i}\oplus\bigoplus_{i=1}^{n_3} H_{2m_i}(\mu_i),$$ for some $\alpha_i,\mu_i\in\C$ with $|\alpha_i|=1$, for all $1\leq i\leq n_2$, and $|\mu_i|>1$, for all $1\leq i\leq n_3$. Then, the {\rm $^*$congruence bundle} of $A$ is the following subset of $\C^{n\times n}$
    \begin{align}\label{bundle*}
        \bunh(A):= &\left\{P\left(\bigoplus_{i=1}^{n_1}J_{k_i}(0)\oplus\bigoplus_{i=1}^{n_2}\widetilde \alpha_i\Gamma_{\ell_i}\oplus\bigoplus_{i=1}^{n_3} H_{2m_i}(\widetilde\mu_i)\right)P^* \right. \nonumber \\  & \hspace*{0.8cm} :\     \left. \begin{array}{l}
             P\in \C^{n\times n} \; \mbox{\rm is invertible},  \\[0.1cm]
              |\widetilde\alpha_i|=1,\ 1\leq i\leq n_2,\\[0.1cm]
              |\widetilde\mu_i|>1,\ 1\leq i\leq n_3,\\[0.1cm]
\widetilde\alpha_i^2\neq\widetilde\alpha_j^2\Longleftrightarrow\alpha_i^2\neq\alpha_j^2, \quad \mbox{if $\ell_i = \ell_j \, ({\rm mod} \, 2)$},
\\[0.1cm]
\widetilde\alpha_i^2\neq -\widetilde\alpha_j^2\Longleftrightarrow\alpha_i^2\neq -\alpha_j^2, \quad \mbox{if $\ell_i \ne \ell_j \, ({\rm mod} \, 2)$},
\\[0.1cm]
\widetilde\mu_i\neq\widetilde\mu_j\Longleftrightarrow\mu_i\neq\mu_j
        \end{array}\right\}.
    \end{align}
\end{defi}

\medskip

In words, the congruence bundle of $A$ is the set of $n\times n$ matrices with the same CFC as $A$, up to the values of the distinct complex numbers $\mu_i \, (\ne \pm 1)$ associated to the Type II-(a) blocks. Analogously, the $^*$congruence bundle of $A$ is the set of $n\times n$ matrices with the same $^*$CFC as $A$, up to the values of the complex numbers $\alpha_i$ and $\mu_i$ associated to the Type I and Type II blocks, as long as these values satisfy the same inequality constraints as in $A$ for all the matrices in the $^*$congruence bundle. It is worth to emphasize that the values $\alpha_i$ and $\mu_i$ for the matrices in the same bundle are arbitrary as long as the inequality constraints in the definitions of bundles are satisfied, as well as the restrictions imposed in Corollary \ref{coro.CFC} and Theorem \ref{hs_th}-(b).

In the last part of this section, we define some subsets of complex numbers that will be used in the rest of the paper and, in particular, in some of the technical lemmas of Section \ref{sec.techlemmas}. Note that Definition \ref{def.auxsets} uses the concept of ``set'' with its standard mathematical meaning, i.e., repetitions (or multiplicities) of elements are not allowed.

\begin{defi} \label{def.auxsets} Let $A\in \C^{n\times n}$.
\begin{enumerate}
    \item[\rm (a)] $\mathcal{S}_{H}^{\top} (A)$ is the set of complex numbers $\mu$ that appear in the Type II-(a) blocks of the CFC of $A$.
    \item[\rm (b)] $\mathcal{S}_{\Gamma}^{*} (A)$ is the set of complex numbers $\alpha$ that appear in the Type I blocks of the $^*$CFC of $A$.
     \item[\rm (c)] $\mathcal{S}_{\Gamma , 2}^{*} (A) := \{\alpha^2 \, : \, \alpha \in \mathcal{S}_{\Gamma}^{*} (A)\}$.
     \item[\rm (d)] $\mathcal{S}_{\Gamma , 2}^{*,neg} (A) := \{-\alpha^2 \, : \, \alpha \in \mathcal{S}_{\Gamma}^{*} (A)\}$.
    \item[\rm (e)] $\mathcal{S}_{H}^{*} (A)$ is the set of complex numbers $\mu$ that appear in the Type II blocks of the $^*$CFC of $A$.
\end{enumerate}

\end{defi}

\section{Technical lemmas on congruence bundles}\label{sec.techlemmas}
In this section we prove several auxiliary technical lemmas on congruence bundles that are needed to prove the main results, i.e., Theorems \ref{main_th} and \ref{th.opennness}. Some of these results can be stated and proved simultaneously for congruence and $^*$congruence bundles. Therefore, from now on, we use the notation $\bunstar(\cdot)$ to include both $\bunc(\cdot)$ and $\bunh(\cdot)$ from Definitions \ref{bundlec_def} and \ref{bundle*_def}. In words, we will refer to $\bunstar(\cdot)$ as $^\star$congruence bundle.

In several of the proofs of this paper, we will make use of the connection between the $^\star$CFC of a matrix $A\in\C^{n\times n}$ and the Kronecker canonical form (KCF) under strict equivalence of the pencil $A+\la A^\star$. The KCF is the canonical form of matrix pencils under {\em strict equivalence}, that is, under multiplication on the right and on the left by different invertible matrices, and is a direct sum of certain blocks (see, for instance, \cite[Ch. XII]{gant}). We have already mentioned that $A,B \in \C^{n\times n}$ are congruent if and only if the pencils $A+ \la A^\top$ and $B+ \la B^\top$ are strictly equivalent. Therefore, it is not surprising that there is a bijection between the blocks of the CFC of $A$ and the blocks of the KCF of $A+\la A^\top$. This bijection can be found, for instance, in \cite[Th. 4]{d16}, which will be referred to in the proofs of several of the results of this paper. In contrast, there is no a bijection between the blocks of the $^*$CFC of $A$ and the blocks of the KCF of $A+\la A^*$. More precisely, the $^*$CFC of $A$ determines the KCF of $A + \la A^*$, but the KCF of $A + \la A^*$ only determines the $^*$CFC of $A$ up to multiplication of any Type I block by $-1$. The detailed relation between such $^*$CFC and KCF can be found in \cite[Th. 2.13-(b)]{dd25}, which will be also referred to in the proofs of several of the results of this paper.

The next lemma allows us to determine if a matrix $A$ is in the closure of the $^\star$congruence bundle of another matrix just by studying whether the $^\star$CFC of $A$ belongs to that closure or not.

\begin{lemma}\label{congruencebundle_lem}
Let $A,B\in\C^{n\times n}$. Then  $A\in\cbunstar(B)$ if and only if $P^\star AP\in\cbunstar(B)$, for every invertible matrix $P$.
\end{lemma}
\begin{proof}
    Assume that $A\in\cbunstar(B)$. Then, there is a sequence of matrices $\{B_k\}_{k\in\mathbb N}$ which converges to $A$ and with $B_k\in\bunstar(B)$, for all $k\in\mathbb N$. Then, the sequence $\{P^\star B_kP\}_{k\in\mathbb N}$ converges to $P^\star AP$ and $P^\star B_kP\in\bunstar(B)$. Therefore, $P^*AP\in\cbunstar(B)$. The other implication is trivial, taking $P=I_n$.
\end{proof}

Lemmas \ref{lemm.dirsumbund1} and \ref{lemm.dirsumbund2} for $^*$congruence bundles, and their counterparts Lemmas \ref{lemm.dirsumbund3} and \ref{lemm.dirsumbund4} for congruence bundles, prove some results about how bundles and their closures interact with direct sums. These results have the same flavor as Lemmas 3.2 and 3.3 in \cite{ddd24}. We will prove in detail Lemmas \ref{lemm.dirsumbund1} and \ref{lemm.dirsumbund2}, which are more involved, and we will just sketch the proofs of Lemmas \ref{lemm.dirsumbund3} and \ref{lemm.dirsumbund4}. These lemmas are of interest because Lemma \ref{congruencebundle_lem} allows us to reduce our problems to matrices which are in $^\star$CFC. Then, our strategy is to study the most generic bundles whose closures contain each of the canonical blocks which are the direct summands of the $^\star$CFC, and, finally, to combine such bundles. In this setting, we will need to choose the matrices that are used to define the bundles in such a way that the closure of the bundle corresponding to its direct sum includes all the possible direct sums of the matrices in the individual bundles. Example \ref{ex.bundles} illustrates this question.

\begin{example} \label{ex.bundles} {\rm
Consider the $^*$congruence bundles $\bunh(H_2(3) \oplus \Gamma_1)$ and $\bunh(H_4(3) \oplus \Gamma_1)$. Then $H_2(3) \oplus \Gamma_1 \in \bunh(H_2(3) \oplus \Gamma_1)$ and $H_4(5) \oplus \left(e^{\ii \pi/4} \Gamma_1 \right) \in \bunh(H_4(3) \oplus \Gamma_1)$, but
$$
H_2(3) \oplus \Gamma_1 \oplus H_4(5) \oplus \left(e^{\ii \pi/4} \Gamma_1 \right) \notin \bunh(H_2(3) \oplus \Gamma_1 \oplus  H_4(3) \oplus \Gamma_1).
$$
Note that $\bunh(H_4(3) \oplus \Gamma_1) = \bunh(H_4(5) \oplus \left(e^{\ii \pi/4} \Gamma_1 \right))$ and that, obviously,
\begin{equation} \label{eq1.exbundle}
H_2(3) \oplus \Gamma_1 \oplus H_4(5) \oplus \left(e^{\ii \pi/4} \Gamma_1 \right) \in \bunh(H_2(3) \oplus \Gamma_1 \oplus  H_4(5) \oplus \left(e^{\ii \pi/4} \Gamma_1 \right) ).
\end{equation}
However, $H_2(3) \oplus \Gamma_1 \oplus H_4(3) \oplus \Gamma_1$ does not belong to the bundle in \eqref{eq1.exbundle}. In fact, it is not difficult to prove that there is no bundle that includes $A_1 \oplus A_2$ for any  $A_1 \in \bunh(H_2(3) \oplus \Gamma_1)$ and for any $A_2 \in \bunh(H_4(3) \oplus \Gamma_1)$. In order to define a mathematical object that includes all such direct sums, we need to use closures of bundles, in particular,
\begin{equation} \label{eq2.exbundle}
\cbunh(H_2(3) \oplus \Gamma_1 \oplus  H_4(5) \oplus \left(e^{\ii \pi/4} \Gamma_1 \right) )
\end{equation}
contains all the desired $A_1 \oplus A_2$. To illustrate this statement consider, for instance, for each $k \in {\mathbb N}$, the pencil $H_2(3) \oplus \Gamma_1 \oplus H_4(3+ 1/k) \oplus \left(e^{\ii /k} \Gamma_1 \right)$, which belongs to $\bunh(H_2(3) \oplus \Gamma_1 \oplus  H_4(5) \oplus \left(e^{\ii \pi/4} \Gamma_1 \right) )$, and, since
$$
\lim_{k\rightarrow \infty} H_2(3) \oplus \Gamma_1 \oplus H_4(3+ 1/k) \oplus \left(e^{\ii /k} \Gamma_1 \right) = H_2(3) \oplus \Gamma_1 \oplus H_4(3) \oplus \Gamma_1,
$$
we get that $H_2(3) \oplus \Gamma_1 \oplus H_4(3) \oplus \Gamma_1 \in \cbunh(H_2(3) \oplus \Gamma_1 \oplus  H_4(5) \oplus \left(e^{\ii \pi/4} \Gamma_1 \right) )$.  The key idea we have used to construct \eqref{eq2.exbundle} is to choose one matrix $C_1\in \bunh(H_2(3) \oplus \Gamma_1)$ and another $C_2\in \bunh(H_4(3) \oplus \Gamma_1)$ such that their corresponding sets in Definition \ref{def.auxsets}(c)-(e) are disjoint. See Lemma \ref{lemm.dirsumbund1} for the general precise conditions.
}
\end{example}

\begin{lemma} \label{lemm.dirsumbund1} If $A_1 \in \C^{p_1 \times p_1} , \ldots , A_q \in \C^{p_q \times p_q}$ and $C_1 \in \C^{p_1 \times p_1} , \ldots , C_q \in \C^{p_q \times p_q}$ satisfy
\begin{enumerate}
    \item[\rm (1)] $A_i \in \bunh (C_i)$, for $i=1, \ldots , q$,

    \item[\rm (2)] $\mathcal{S}_{\Gamma,2}^* (C_i) \cap \mathcal{S}_{\Gamma,2}^* (C_j) = \emptyset$ and
                  $\mathcal{S}_{\Gamma,2}^* (C_i) \cap \mathcal{S}_{\Gamma,2}^{*,neg} (C_j) = \emptyset$ for $i \ne j$, $i,j=1, \ldots q$,
    \item[\rm (3)] $\mathcal{S}_{H}^* (C_i) \cap \mathcal{S}_{H}^* (C_j) = \emptyset$ for $i \ne j$, $i,j=1, \ldots q$,
\end{enumerate}
 then $$A_1 \oplus \cdots \oplus A_q \in \cbunh (C_1 \oplus \cdots \oplus C_q).$$ Moreover, if $A_1, \ldots , A_q$ satisfy the conditions {\rm (2)} and {\rm (3)} (with $A_i,A_j$ instead of $C_i,C_j$), then $A_1 \oplus \cdots \oplus A_q \in \bunh (C_1 \oplus \cdots \oplus C_q)$.
\end{lemma}
\begin{proof} {\em Case} 1. If $A_1, \ldots , A_q$ satisfy conditions {\rm (2)} and {\rm (3)}, i.e., if $\mathcal{S}_{\Gamma,2}^* (A_i) \cap \mathcal{S}_{\Gamma,2}^* (A_j) = \emptyset$, $\mathcal{S}_{\Gamma,2}^* (A_i) \cap \mathcal{S}_{\Gamma,2}^{*,neg} (A_j) = \emptyset$, and $\mathcal{S}_{H}^* (A_i) \cap \mathcal{S}_{H}^* (A_j) = \emptyset$ for $i \ne j$, $i,j=1, \ldots q$, then Definition \ref{bundle*_def} and the fact that the $^*$CFC of a direct sum is equal to the direct sum of the $^*$CFCs of the direct summands imply that $A_1 \oplus \cdots \oplus A_q \in \bunh (C_1 \oplus \cdots \oplus C_q)$.

 {\em Case} 2. Assume that $A_1, \ldots , A_q$ do not satisfy conditions {\rm (2)} or {\rm (3)}. Let us define the following two multisets\footnote{Multiset is a modification of the concept of set that, unlike a set, allows for multiple instances for each of its elements. The number of instances of an element is called the multiplicity of that element in the multiset.} of complex numbers
 $$
 [\alpha_1, \ldots, \alpha_s] := \left\{
 \begin{array}{l}
\alpha \, : \begin{array}{l}\, \alpha \in  \mathcal{S}_{\Gamma}^* (A_i) \; \ \mbox{and}\\ \alpha^2 \in  \left( \mathcal{S}_{\Gamma,2}^* (A_i) \cap \mathcal{S}_{\Gamma,2}^* (A_j) \right) \cup \left( \mathcal{S}_{\Gamma,2}^* (A_i) \cap \mathcal{S}_{\Gamma,2}^{*,neg} (A_j) \right) \; \mbox{for some $i \ne j$}.\end{array} \\
\mbox{If $\alpha \in  \mathcal{S}_{\Gamma}^* (A_i)$\ for exactly $\ell$ indices $i = 1, \ldots, q$, the multiplicity of $\alpha$ is $\ell$}
\end{array} \right\}
 $$
and
$$
 [\mu_1, \ldots, \mu_t] := \left\{
 \begin{array}{l}
\mu \, : \,  \mu \in  \mathcal{S}_{H}^* (A_i) \cap \mathcal{S}_{H}^* (A_j) \; \mbox{for some $i \ne j$}.\\
\mbox{If $\mu \in  \mathcal{S}_{H}^* (A_i)$ for exactly $\ell$ indices $i = 1, \ldots, q$, the multiplicity of $\mu$ is $\ell$}
\end{array}\  \right\}.
$$
From these multisets, we define, for each $k\in \mathbb{N}$, the following multisets
\begin{align*}
\mathcal{A}_k &:= \left[\alpha_1 \exp \left(\frac{\ii}{k}\right), \ldots, \alpha_s \exp\left( \frac{\ii}{sk} \right) \right], \\ \mathcal{M}_k & :=\left[\mu_1 \exp\left(\frac{\ii}{k}\right), \ldots, \mu_t \exp\left(\frac{\ii}{tk}\right)\right].
\end{align*}
Then, for $k$ large enough, all the elements in  $\mathcal{A}_k$ are distinct and all the elements of $\mathcal{M}_k$ are  distinct as well. To see this for $\mathcal{A}_k$, note first that, if $\alpha_{j}=\alpha_{j'}$, for $1\leq j\neq j'\leq s$, then $\alpha_{j}\exp(\ii/(jk))\neq\alpha_{j'}\exp(\ii/(j'k))$, for all $k$. If $\alpha_{j}\neq\alpha_{j'}$, then $\alpha_{j}\exp(\ii/(jk))=\alpha_{j'}\exp(\ii/(j'k))$ if and only if $\alpha_j/\alpha_{j'}=\exp((\ii/k)\cdot(1/j'-1/j))$. If we set $\alpha_j=e^{\ii\theta_j}$ and $\alpha_{j'} =e^{\ii\theta_{j'}}$, for $0\leq\theta_j,\theta_{j'}<2\pi$, this is equivalent to $\theta_j-\theta_{j'}=(1/k)\cdot(1/j'-1/j)$. Taking $k$ large enough so that $(1/k)\cdot|1/j'-1/j|<|\theta_j-\theta_{j'}|$, the last identity does not hold. The reasoning for $\mathcal{M}_k$ is similar.

Next, let $A_i = P_i G_i P_i^*$ with $G_i$ being the $^*$CFC of $A_i$ and $P_i$ invertible for $i= 1, \ldots q$, and define $q$ sequences of matrices
$$
\{A_1^{(k)}\}_{k \in \mathbb{N}} := \{ P_1 G_1^{(k)} P_1^* \}_{k \in \mathbb{N}} \, , \; \ldots \; ,
\{A_q^{(k)}\}_{k \in \mathbb{N}} := \{ P_q G_q^{(k)} P_q^* \}_{k \in \mathbb{N}} \, ,
$$
where $G_1^{(k)}, \ldots , G_q^{(k)}$ are obtained from $G_1, \ldots , G_q$ by replacing the parameters $[\alpha_1, \cdots, \alpha_s]$ and $[\mu_1, \cdots, \mu_t]$ by those in $\mathcal{A}_k$ and $\mathcal{M}_k$, respectively, and keeping the other elements in $\mathcal{S}_\Gamma^* (G_1), \ldots ,$ $\mathcal{S}_\Gamma^* (G_q)$ and  $\mathcal{S}_H^* (G_1), \ldots ,  \mathcal{S}_H^* (G_q)$ unchanged. Note that:
\begin{enumerate}
    \item $A_i^{(k)} \in \bunh (C_i)$, for $i=1, \ldots , q$, and for $k$ large enough,
    \item $\mathcal{S}_{\Gamma,2}^* (A_i^{(k)}) \cap \mathcal{S}_{\Gamma,2}^* (A_j^{(k)}) = \emptyset$, $\mathcal{S}_{\Gamma,2}^* (A_i^{(k)}) \cap \mathcal{S}_{\Gamma,2}^{*,neg} (A_j^{(k)}) = \emptyset$, and $\mathcal{S}_{H}^* (A_i^{(k)}) \cap \mathcal{S}_{H}^* (A_j^{(k)}) = \emptyset$ for $i \ne j$, $i,j=1, \ldots , q$, and for $k$ large enough.
\end{enumerate}
Thus, by Case 1,  $A_1^{(k)} \oplus \cdots \oplus A_q^{(k)} \in \bunh (C_1 \oplus \cdots \oplus C_q)$ for $k$ large enough, and, since,
$$
\lim_{k\rightarrow \infty}  A_1^{(k)} \oplus \cdots \oplus A_q^{(k)} = A_1 \oplus \cdots \oplus A_q,
$$
we get that $A_1 \oplus \cdots \oplus A_q \in \cbunh (C_1 \oplus \cdots \oplus C_q).$
\end{proof}

Lemma \ref{lemm.dirsumbund2} is more general than Lemma \ref{lemm.dirsumbund1} and follows from it. The only difference is in condition (1), where we replace $\bun^*(C_i)$ by $\overline{\bun^*}(C_i)$. Lemma \ref{lemm.dirsumbund2} is the one that will be used in the proof of part (b) in Theorem \ref{main_th}.

\begin{lemma} \label{lemm.dirsumbund2} If $A_1 \in \C^{p_1 \times p_1} , \ldots , A_q \in \C^{p_q \times p_q}$ and $C_1 \in \C^{p_1 \times p_1} , \ldots , C_q \in \C^{p_q \times p_q}$ satisfy
\begin{enumerate}
    \item[\rm (1)] $A_i \in \cbunh (C_i)$, for $i=1, \ldots , q$,

    \item[\rm (2)] $\mathcal{S}_{\Gamma,2}^* (C_i) \cap \mathcal{S}_{\Gamma,2}^* (C_j) = \emptyset$ and
                  $\mathcal{S}_{\Gamma,2}^* (C_i) \cap \mathcal{S}_{\Gamma,2}^{*,neg} (C_j) = \emptyset$ for $i \ne j$, $i,j=1, \ldots q$,
    \item[\rm (3)] $\mathcal{S}_{H}^* (C_i) \cap \mathcal{S}_{H}^* (C_j) = \emptyset$ for $i \ne j$, $i,j=1, \ldots q$,
\end{enumerate}
 then $$A_1 \oplus \cdots \oplus A_q \in \cbunh (C_1 \oplus \cdots \oplus C_q).$$
\end{lemma}
\begin{proof}
Condition (1) ensures that there exist $q$ sequences of matrices
$
\{A_1^{(k)}\}_{k \in \mathbb{N}} \subset \bunh (C_1) ,  \ldots ,$
$
\{A_q^{(k)}\}_{k \in \mathbb{N}} \subset \bunh (C_q)
$
such that $\lim_{k\rightarrow \infty} A_1^{(k)} = A_1, \ldots , \lim_{k\rightarrow \infty} A_q^{(k)} = A_q$. Lemma \ref{lemm.dirsumbund1} applied to
$
A_1^{(k)} ,  \ldots , A_q^{(k)}
$
implies that $A_1^{(k)} \oplus \cdots \oplus A_q^{(k)} \in \cbunh (C_1 \oplus \cdots \oplus C_q)$ for all $k\in \mathbb{N}$. So,
$$
A_1 \oplus \cdots \oplus A_q  = \lim_{k\rightarrow \infty} (A_1^{(k)} \oplus \cdots \oplus A_q^{(k)}) \in \cbunh (C_1 \oplus \cdots \oplus C_q).
$$
\end{proof}

Lemmas \ref{lemm.dirsumbund3} and  \ref{lemm.dirsumbund4} for congruence bundles are simpler than Lemmas \ref{lemm.dirsumbund1} and  \ref{lemm.dirsumbund2} for $^*$congru\-ence bundles because the CFC in Corollary \ref{coro.CFC} does not involve parameters associated to the Type I blocks.
\begin{lemma} \label{lemm.dirsumbund3} If $A_1 \in \C^{p_1 \times p_1} , \ldots , A_q \in \C^{p_q \times p_q}$ and $C_1 \in \C^{p_1 \times p_1} , \ldots , C_q \in \C^{p_q \times p_q}$ satisfy
\begin{enumerate}
    \item[\rm (1)] $A_i \in \bunc (C_i)$, for $i=1, \ldots , q$,

    \item[\rm (2)] $\mathcal{S}_{H}^\top (C_i) \cap \mathcal{S}_{H}^\top (C_j) = \emptyset$ for $i \ne j$, $i,j=1, \ldots q$,
\end{enumerate}
then $$A_1 \oplus \cdots \oplus A_q \in \cbunc (C_1 \oplus \cdots \oplus C_q).$$
Moreover, if $A_1, \ldots , A_q$ satisfy the conditions {\rm (2)} (with $A_i,A_j$ instead of $C_i,C_j$), then $A_1 \oplus \cdots \oplus A_q \in \bunc (C_1 \oplus \cdots \oplus C_q)$.
\end{lemma}
\begin{proof}
The proof is similar to that of Lemma \ref{lemm.dirsumbund1} with the simplification that in Case 2 there is only one multiset involved, instead of two. Such a multiset is
$$
[\mu_1, \ldots, \mu_t] := \left\{\begin{array}{l}\mu \, : \,
 \begin{array}{l}
\mu \in  \mathcal{S}_{H}^\top (A_i) \cap \mathcal{S}_{H}^\top (A_j) \; \mbox{for some $i \ne j$}. \\
\end{array}\\
\mbox{If $\mu\in \mathcal{S}_{H}^\top (A_i)$ for exactly $\ell$ indices $i = 1, \ldots, q$, the multiplicity of $\mu$ is $\ell$}\end{array}\right\}.
$$
The rest of the details are omitted.
\end{proof}

The proof of Lemma \ref{lemm.dirsumbund4} is exactly as that of Lemma \ref{lemm.dirsumbund2}, with the obvious change of using Lemma \ref{lemm.dirsumbund3} instead of Lemma \ref{lemm.dirsumbund1}. Therefore, it is omitted.
\begin{lemma} \label{lemm.dirsumbund4} If $A_1 \in \C^{p_1 \times p_1} , \ldots , A_q \in \C^{p_q \times p_q}$ and $C_1 \in \C^{p_1 \times p_1} , \ldots , C_q \in \C^{p_q \times p_q}$ satisfy
\begin{enumerate}
    \item[\rm (1)] $A_i \in \cbunc (C_i)$, for $i=1, \ldots , q$,

    \item[\rm (2)] $\mathcal{S}_{H}^\top (C_i) \cap \mathcal{S}_{H}^\top (C_j) = \emptyset$ for $i \ne j$, $i,j=1, \ldots q$,
\end{enumerate}
then $$A_1 \oplus \cdots \oplus A_q \in \cbunc (C_1 \oplus \cdots \oplus C_q).$$
\end{lemma}

\medskip

Lemma \ref{decomp_lem} is another key ingredient of the proof of the main Theorem \ref{main_th}. In Lemma \ref{decomp_lem}, we show that, for every Type I and Type II blocks in Theorem \ref{hs_th}, there are arbitrarily small perturbations that make them $^\star$congruent to either a direct sum of Type II blocks with size $2\times 2$ and with all their parameters distinct, or to a direct sum of this kind of blocks plus an additional Type I block with size $1\times 1$. For the sake of brevity, in some of the following proofs, if $z\in\C$, we denote by $z^\star$ either $\overline z$ or $z$ depending on whether $\star=*$ or $\star=\top$, respectively.
\begin{lemma}\label{decomp_lem}
    \begin{itemize}
        \item[\rm(i)] Let $\alpha \in \C$ be such that $|\alpha| = 1$ if $\star = *$, or $\alpha = 1$ if $\star = \top$, and let $k>1$. Then $\alpha\Gamma_k\in\overline\bunstar(M)$, where
        \begin{equation*}
        M=\left\{\begin{array}{lc}
             \bigoplus_{i=1}^{k/2}H_2(\mu_i)&\mbox{if $k$ is even,}  \\[0.2cm]
             \bigoplus_{i=1}^{\lfloor k/2\rfloor}H_2(\mu_i)\oplus\alpha &\mbox{if $k$ is odd},
        \end{array}\right.
        \end{equation*}
        with  $\mu_i\neq\mu_{i'}$ for $i\neq i'$, and $|\mu_i| >1$ if $\star = *$, or $\mu_i$ satisfying the conditions in {\rm Type II-(a)} blocks of Corollary {\rm\ref{coro.CFC}} if $\star = \top$, for all $i, i'=1,\hdots,\lfloor\frac{k}{2}\rfloor$.

        \item[\rm(ii)] Let $\mu\in\C$ be such that $|\mu|>1$ if $\star =*$, or $0\ne \mu \ne (-1)^{k+1}$ if $\star =\top$. Then $H_{2k}(\mu)\in\overline\bunstar\left(\bigoplus_{i=1}^kH_2(\mu_i)\right)$, with  $\mu_i\neq\mu_{i'}$ for $i\neq i'$, and $|\mu_i| >1$ if $\star = *$, or $\mu_i$ satisfying the conditions in {\rm Type II-(a)} blocks of Corollary {\rm\ref{coro.CFC}} if $\star = \top$, for all $i, i'=1,\hdots,k$.
    \end{itemize}
\end{lemma}
\begin{proof}
(i) Assume first that $k$ is even. Let $\varepsilon_1,\varepsilon_2,\cdots,\varepsilon_k$ be positive real numbers and set
$$
\Gamma_k(\varepsilon_1,\varepsilon_2,\hdots,\varepsilon_k):=\begin{bmatrix}
    &&&&&&&&-\varepsilon_2\\
    &&&&&&&\varepsilon_4\\
    &&&&&&\iddots&&\\
    &&&&(-1)^{k/2}\varepsilon_k&&&\\
    &&&(-1)^{k/2-1}\varepsilon_{k-1}\\
    &&\iddots&&\\
    &-\varepsilon_3\\
    \varepsilon_1
\end{bmatrix}.
$$
Let us consider the $\star$-palindromic matrix pencil
$$
\begin{array}{ccl}
L_{\alpha,k}(\la)&:=&\alpha(\Gamma_k+\Gamma_k(\varepsilon_1,\varepsilon_2,\hdots,\varepsilon_k))+\la \,(\alpha(\Gamma_k+\Gamma_k(\varepsilon_1,\varepsilon_2,\hdots,\varepsilon_k)))^\star\\[0.2cm]
&=&\begin{bmatrix}
    &&&&L_{1,k}(\la)\\
    &&&L_{2,k-1}(\la)&\alpha+\alpha^\star\la\\
    &&\iddots&\iddots&\\
    &L_{k-1,2}(\la)&-\alpha-\alpha^\star\la\\
    L_{k,1}(\la)&\alpha+\alpha^\star\la
\end{bmatrix}_{k\times k} ,
\end{array}
$$
where $L_{j,k-j+1}(\la)=(-1)^j\alpha(1+\varepsilon_{2j})+(-1)^{j-1}\la\alpha^\star(1+\varepsilon_{2j-1})$, for $j=1,\hdots,k/2$, and $L_{j,k-j+1}(\la)=(-1)^{k-j}\alpha(1+\varepsilon_{2(k-j)+1})+(-1)^{k-j+1}\la\alpha^\star(1+\varepsilon_{2(k-j+1)})$, for $j=k/2+1,\hdots,k$. Note that $\la L^\star_{j,k-j+1}(1/\la)=L_{k-j+1,j}(\la)$, for $j=1,\hdots,k/2$. Hence, the eigenvalues of $L(\la)$ are $\la_j =\alpha(1+\varepsilon_{2j})/(\alpha^\star(1+\varepsilon_{2j-1}))$ and $1/\la_j^\star=\alpha(1+\varepsilon_{2j-1})/(\alpha^\star(1+ \varepsilon_{2j}))$, for $j=1,\hdots,k/2$. If $\varepsilon_1,\hdots,\varepsilon_k$ are such that
$(1+\varepsilon_{2i})/(1+\varepsilon_{2i-1})\neq(1+\varepsilon_{2j})/(1+\varepsilon_{2j-1})$, for $i\neq j$, and $(1+\varepsilon_{2i})/(1+\varepsilon_{2i-1}) >1$, for $i,j = 1,\ldots , k/2$ (this holds, for instance, if $\varepsilon_j = j \delta$ for any $\delta >0$), then these eigenvalues are all distinct numbers with absolute values different from $1$, so the KCF of $L_{\alpha,k}(\la)$ is $\bigoplus_{i=1}^{k/2}(\la-\la_i)\oplus\bigoplus_{i=1}^{k/2}(\la-1/\la_i^\star)$. This implies, by \cite[Th. 4]{d16} if $\star = \top$ or by \cite[Th. 2.13-(b)]{dd25} if $\star = *$, that the $^\star$congruence canonical form of $\alpha(\Gamma_k+\Gamma_k(\varepsilon_1,\varepsilon_2,\hdots,\varepsilon_k))$ is $\bigoplus_{i=1}^{k/2}H_2(-\la_i)$. As a consequence, $\alpha(\Gamma_k+\Gamma_k(\varepsilon_1,\varepsilon_2,\hdots,\varepsilon_k))\in\bunstar\left(\bigoplus_{i=1}^{k/2}H_2(\mu_i)\right)$, with $\mu_i$ as in the statement. Since $\Gamma_k+\Gamma_k(\varepsilon_1,\varepsilon_2,\hdots,\varepsilon_k)$ converges to $\Gamma_k$ as $\varepsilon_1,\varepsilon_2,\hdots,\varepsilon_k$ tend to $0$, we conclude that $\alpha\Gamma_k\in\cbunstar\left(\bigoplus_{i=1}^{k/2}H_2(\mu_i)\right)$, as desired.

\medskip

Now, assume that $k>1$ is odd, and let $\varepsilon_1,\hdots,\varepsilon_{k-1}$ be positive real numbers, and set
$$
\widehat\Gamma_k(\varepsilon_1,\varepsilon_2,\hdots,\varepsilon_{k-1}):=\begin{bmatrix}
    &&&&&&&&&-\varepsilon_2\\
   &&&&&&&& \varepsilon_4\\
    &&&&&&&\iddots&&\\
    &&&&&(-1)^{(k-1)/2}\varepsilon_{k-1}&&&\\&&&&0&&&\\
    &&&(-1)^{(k-1)/2-1}\varepsilon_{k-2}\\
    &&\iddots&&\\
    &-\varepsilon_3\\
    \varepsilon_1
\end{bmatrix}.
$$
Consider the $\star$-palindromic pencil $\widehat L_{\alpha,k}(\la):=\alpha(\Gamma_k+ \widehat\Gamma_k(\varepsilon_1,\varepsilon_2,\hdots,\varepsilon_{k-1}))+\la(\alpha(\Gamma_k+\widehat\Gamma_k(\varepsilon_1,\varepsilon_2,\hdots,\varepsilon_{k-1})))^\star$, which is of the form
$$
\widehat L_{\alpha,k}(\la)=\begin{bmatrix}
    &&&&\widehat L_{1,k}(\la)\\
    &&&\widehat L_{2,k-1}(\la)&-\alpha+\alpha^\star\la\\
    &&\iddots&\iddots&\\
    &\widehat L_{k-1,2}(\la)&-\alpha+\alpha^\star\la\\
    \widehat L_{k,1}(\la)&\alpha-\alpha^\star\la
\end{bmatrix}_{k\times k},
$$
where $\widehat L_{j,k-j+1}(\la)=(-1)^j\alpha(-1+\varepsilon_{2j})+(-1)^{j-1}\la\alpha^\star(1+\varepsilon_{2j-1})$, for $j=1,\hdots,(k-1)/2$, $\widehat L_{j,k-j+1}(\la)=(-1)^{k-j}\alpha(1+\varepsilon_{2(k-j)+1})+(-1)^{k-j+1}\la\alpha^\star(-1+\varepsilon_{2(k-j+1)})$, for $j=(k+1)/2+1,\hdots,k$, and $\widehat L_{(k+1)/2,(k+1)/2}(\la)= (-1)^{(k-1)/2} (\alpha+\alpha^\star\la)$. Again, $\la \widehat L^\star_{j,k-j+1}(1/\la)=\widehat L_{k-j+1,j}(\la)$, for $j=1,\hdots,(k+1)/2$. Then, the eigenvalues of $\widehat L_{\alpha,k}(\la)$ are $\la_j=\alpha(-1+\varepsilon_{2j})/(\alpha^\star(1+\varepsilon_{2j-1}))$, for $j=1,\hdots,(k-1)/2$, together with $1/\la_j^\star$ and $-\alpha/\alpha^\star$. If $\varepsilon_1,\hdots,\varepsilon_{k-1}$ are such that $(-1+\varepsilon_{2i})/(1+\varepsilon_{2i-1})\neq(-1+\varepsilon_{2j})/(1+\varepsilon_{2j-1})$, for all $i\neq j$, and $0<|(-1+\varepsilon_{2i})/(1+\varepsilon_{2i-1})|<1$, for $i,j = 1, \ldots, (k-1)/2$ (this holds, for instance, if $\varepsilon_j = j \delta$ for any $1/k > \delta >0$), then these eigenvalues are all distinct numbers with absolute values different from $1$ except for the eigenvalue $-\alpha/\alpha^\star$. Therefore, the KCF of $\widehat L_{\alpha,k}(\la)$ is $\bigoplus_{i=1}^{(k-1)/2}(\la-\la_i)\oplus\bigoplus_{i=1}^{(k-1)/2}(\la-1/\la_i^\star)\oplus (\la + \alpha/\alpha^\star ) $. This implies, taking into account that $|1/\la^\star| >1$ and using again \cite[Th. 4]{d16} if $\star = \top$ or \cite[Th. 2.13-(b)]{dd25} if $\star = *$, that the $^\star$congruence canonical form of $\alpha(\Gamma_k+\widehat\Gamma_k(\varepsilon_1,\varepsilon_2,\hdots,\varepsilon_{k-1}))$ is $\bigoplus_{i=1}^{(k-1)/2}H_2(-1/\la_i^\star)\oplus \beta$, with $\beta = 1$ if $\star = \top$, or $|\beta| =1$ and $\beta / \overline \beta = \alpha/\overline \alpha$ if $\star = *$. Reasoning as in the case $k$ even and taking into account that the parameters in the matrices chosen to define a bundle are arbitrary except for the imposed constraints, we conclude that $\alpha\Gamma_k\in\cbunstar\left(\bigoplus_{i=1}^{(k-1)/2}H_2(\mu_i)\oplus\alpha\right)$ as desired.

\medskip

(ii) If $k=1$ the result is obvious. Therefore we assume that $ k>1$. If $\star = \top$, we assume without loss of generality that $\mu$ satisfies the properties in Type II-(a) blocks of Corollary \ref{coro.CFC} or $\mu = (-1)^k$. Thus, for $\star = \top$ or for $\star = *$, any of the considered $\mu$ can expressed as $\mu = |\mu| e^{\ii \phi}$, with $|\mu| \geq 1$ and $\phi \in \mathbb{R}$. Let $F_{2k}(\varepsilon_1,\hdots,\varepsilon_k):=\left[\begin{smallmatrix}
    0&0_{k\times k}\\D(\varepsilon_1,\hdots,\varepsilon_k)&0
\end{smallmatrix}\right]$, with $D(\varepsilon_1,\hdots,\varepsilon_k):=\diag(\varepsilon_1 e^{\ii \phi},\hdots,\varepsilon_k e^{\ii \phi})$ and $\varepsilon_1>\cdots>\varepsilon_k>0$. Then the $\star$-palindromic pencil $H(\la):=H_{2k}(\mu)+F_{2k}(\varepsilon_1,\hdots,\varepsilon_k)+\la(H_{2k}(\mu)+F_{2k}(\varepsilon_1,\hdots,\varepsilon_k))^\star$ has eigenvalues $\la_i=- (|\mu| + \varepsilon_i) \, e^{\ii \phi}$ and $1/\la_i^\star$, for $i=1,\hdots,k$. These $2k$ eigenvalues are all distinct complex numbers and $|\lambda_i| >1$ for $i=1, \ldots, k$, so the KCF of $H(\la)$ is $\bigoplus_{i=1}^k(\la-\la_i)\oplus\bigoplus_{i=1}^k(\la-1/\la_i^\star)$, and, by \cite[Th. 4]{d16} if $\star = \top$ or by \cite[Th. 2.13-(b)]{dd25} if $\star = *$, the $^\star$congruence canonical form of $H_{2k}(\mu)+F_{2k}(\varepsilon_1,\hdots,\varepsilon_k)$ is $\bigoplus_{i=1}^kH_{2}(-\la_i)$. This means that $H_{2k}(\mu)+F_{2k}(\varepsilon_1,\hdots,\varepsilon_k)\in\bunstar(\bigoplus_{i=1}^kH_{2}(\mu_i))$, with $\mu_i$ as in the statement. Since $H_{2k}(\mu)+F_{2k}(\varepsilon_1,\hdots,\varepsilon_k)$ converges to $H_{2k}(\mu)$ as $\varepsilon_1,\hdots,\varepsilon_k$ tend to $0$, we conclude that $H_{2k}(\mu)\in\cbunstar(\bigoplus_{i=1}^kH_{2}(\mu_i))$.
\end{proof}


\medskip

Lemma \ref{mod1evals_lem} is the last lemma of this section and is another key ingredient of the proof of the main Theorem \ref{main_th} for the case $\star = *$. Lemma \ref{mod1evals_lem} gives a lower bound on the number of unit eigenvalues of the $*$-palindromic pencil $A+\la A^*$ when $A$ is invertible and belongs to the closure of the bundle of a direct sum of Type II blocks with size $2\times 2$ and Type I blocks with size $1\times 1$ in the $^*$CFC of Theorem \ref{hs_th}-(b).

\begin{lemma}\label{mod1evals_lem}
Let $A\in\C^{n\times n}$ be invertible and $\ell$ be an integer such that  $0 \leq \ell \leq \lfloor n/2\rfloor$. If
$$
A\in\break\cbunh\left(\bigoplus_{i=1}^\ell H_2(\mu_i)\oplus\bigoplus_{j=1}^{n-2\ell}\alpha_j\right),
$$
with $|\mu_i| >1$, for $i=1, \hdots, \ell$, and $|\alpha_j|=1$, for $j=1,\hdots,n-2\ell$, then the pencil $A+\la A^*$ has, at least, $n-2\ell$ unit eigenvalues (counting multiplicities).
\end{lemma}
\begin{proof} Since
$A\in\cbunh\left(\bigoplus_{i=1}^\ell H_2(\mu_i)\oplus\bigoplus_{j=1}^{n-2\ell}\alpha_j\right)$, there is a sequence $\{A_k\}_{k\in\mathbb N}$, with $A_k\in\C^{n\times n}$, such that
\begin{itemize}
    \item[(i)] $A_k\in\bunh\left(\bigoplus_{i=1}^\ell H_2(\mu_i)\oplus\bigoplus_{j=1}^{n-2\ell}\alpha_j\right)$, and
    \item[(ii)] $\{A_k\}_{k\in\mathbb N}$ converges to $A$.
\end{itemize}
By (i), there is an invertible matrix $P_k$ such that $P_k^* A_kP_k=\bigoplus_{i=1}^\ell H_2(\mu_i^{(k)})\oplus\bigoplus_{j=1}^{n-2\ell}\alpha_j^{(k)}$, with $|\mu_i^{(k)}| >1$, for $i=1,\hdots , \ell$, and $|\alpha_j^{(k)}|=1$, for $j=1,\hdots,n-2\ell$, and for all $k\in\mathbb N$. Then, the pencil $P_k^*(A_k+\la A^*_k)P_k$ and, as a consequence, the pencil $A_k+\la A_k^*$ as well, has $-\alpha_j^{(k)}/\overline{\alpha_j^{(k)}}$ as some of its eigenvalues, for $j=1,\hdots,n-2\ell$. All these complex numbers have modulus $1$. Observe also that all the pencils  $A_k+\la A_k^*$ are regular (see \cite[Th. 2.13-(b)]{dd25}).  By (ii), the sequence of pencils $\{A_k+\la A_k^*\}_{k\in\mathbb N}$ converges to $A+\la A^*$, which is also a regular pencil because $A$ is invertible. By the continuity of eigenvalues of regular matrix pencils (see, for instance, Th. 2.1 in \cite[Ch. VI]{stewartsun}), each sequence  $\{-\alpha_j^{(k)}/\overline{\alpha_j^{(k)}}\}_{k\in\mathbb N}$ of eigenvalues of $A_k+\la A_k^*$  converges to an eigenvalue $\alpha_j$ of $A+\la A^*$, which must have modulus $1$, since all values in the sequence have modulus $1$. Note that some of these $n - 2 \ell$ sequences, say $s$ of them, can converge to the same eigenvalue, say $\alpha_j$, but in this case the eigenvalue $\alpha_j$ will have multiplicity at least $s$, so that $A+\la A^*$ will have, at least, $n-2\ell$ unit eigenvalues, counting multiplicities, as claimed.
\end{proof}

\section{Density of the generic bundles in $\C^{n\times n}$ }\label{main_sec}

In this section, we prove the first of the two main results of this work, which states that the set of $n\times n$ complex matrices is the closure of only one congruence bundle (Theorem \ref{main_th}-(a)), and, also, the union of the closures of $\lfloor n/2\rfloor +1$  $^*$congruence bundles (Theorem \ref{main_th}-(b)). Therefore, the $^\star$CFCs associated to these bundles can be considered as the generic $^\star$CFCs and we see that there is only one generic CFC, which has slightly different expressions for $n$ even or odd, and $\lfloor n/2\rfloor +1$ generic $^*$CFCs. We emphasize that Theorem \ref{main_th}-(b2) proves that all of the $\lfloor n/2\rfloor +1$ generic $^*$CFCs are necessary.

\begin{theorem}\label{main_th}
\begin{itemize}
    \item[\rm(a)]
    For $n\geq 1$, let us define the following congruence bundle
   \begin{equation}\label{mainidentityc}
       \mathcal{G} := \left\{\begin{array}{ll} \bunc\left(\bigoplus_{i=1}^{n/2} H_2(\mu_i)\right),&\mbox{if $n$ is even},\\
        \bunc\left(\bigoplus_{i=1}^{\lfloor n/2\rfloor} H_2(\mu_i)\oplus 1\right),&\mbox{if $n$ is odd},
       \end{array}\right.
   \end{equation}
   where $|\mu_i| >1$ or $\mu_i = e^{i\theta}$, with $0<\theta < \pi$, and $\mu_i\neq\mu_{i'}$, if $i\neq i'$, for all $i,i'=1,\hdots,\lfloor n/2\rfloor$. Then, $\C^{n\times n} = \overline{\mathcal{G}}$.
    \item[\rm(b)] 
    For $n\geq1$, let $\ell = 0, 1, \ldots , \lfloor n/2\rfloor$ and for each of these values $\ell$ let us define the following $^*$congruence bundle
    \begin{equation} \label{def.gen*bund}
    \mathcal{G}_\ell := \bunh\left(\bigoplus_{i=1}^\ell H_2(\mu_i)\oplus\bigoplus_{j=1}^{n-2\ell}\alpha_j\right),
    \end{equation}
     where $|\mu_i|>1$ and $\mu_i\neq\mu_{i'}$, if $i\neq i'$, for all $i,i' =1,\hdots,\ell$, and $|\alpha_j|=1$ and $\alpha_j^2 \neq \alpha_{j'}^2$, if $j\neq j'$, for all $j,j' = 1,\hdots,n-2\ell$. Then,
     \begin{itemize}
         \item[\rm (b1)] $\displaystyle \C^{n\times n}=\bigcup_{\ell=0}^{\lfloor n/2\rfloor} \overline{\mathcal{G}_\ell}$.
         \item[\rm (b2)] $\displaystyle \overline{\mathcal{G}_{\ell_1}} \cap \mathcal{G}_{\ell_2} = \emptyset$ if $0\leq \ell_1\neq\ell_2\leq \lfloor\frac{n}{2}\rfloor$.
     \end{itemize}
   In particular, all the closures of bundles in the union of the right-hand side in {\rm (b1)} are different from each other.
\end{itemize}
\end{theorem}
\begin{proof}
Let us first prove (b1). For this, we are going to prove that, given a matrix $A\in\C^{n\times n}$, there is some $0\leq\ell\leq\lfloor n/2\rfloor$ such that $A\in \cbunh\left(\bigoplus_{i=1}^\ell H_2(\mu_i)\oplus\bigoplus_{j=1}^{n-2\ell}\alpha_j\right)$, with $\mu_i$ and $\alpha_j$ as in the statement.

First, assume that $A$ is invertible. Then, by Theorem \ref{hs_th}-(b), $A$ is $^*$congruent to a direct sum of the form $\bigoplus_{i=1}^{r}H_{2k_i}(\la_i)\oplus\bigoplus_{j=1}^s \beta_j\Gamma_{2m_j}\oplus\bigoplus_{k=1}^t\gamma_k \Gamma_{2n_{k}+1}$, for some $\la_i,\beta_j,\gamma_k\in\C$ with $|\la_i|>1$, for all $i=1,\hdots,r$, and $|\beta_j|=1=|\gamma_k|$, for all $j=1,\hdots,s$ and $k =1,\hdots,t$, and with $2(k_1+\cdots+k_r)+2(m_1+\cdots+m_s)+2(n_1+\cdots+n_t)+t=n$. Let us set, for brevity, $\kappa:=k_1+\cdots+k_r$, $\sigma:=m_1+\cdots+m_s,$ and $\nu:=n_1+\cdots+n_t$, so that $2(\kappa+\sigma+\nu)+t=n$.

Now, by Lemma \ref{decomp_lem}, $H_{2k_i}(\la_i)\in\cbunh(\bigoplus_{i'=1}^{k_{i}}H_2(\mu_{i'}))$, $\beta_j\Gamma_{2m_j}\in\cbunh(\bigoplus_{j'=1}^{m_j}H_2(\mu'_{j'}))$, and $\gamma_k \Gamma_{2n_{k}+1}\in\cbunh(\bigoplus_{k'=1}^{n_k} H_2(\mu''_{k'})\oplus\alpha_k)$, with the properties of the parameters stated in Lemma \ref{decomp_lem}. Moreover, since the specific values of the parameters of a matrix in $^*$CFC used to generate a bundle are arbitrary, we can choose the parameters in the previous bundles to satisfy the conditions of the matrices $C_j$ in Lemma \ref{lemm.dirsumbund2}. Therefore, by Lemma \ref{lemm.dirsumbund2},
$$
\bigoplus_{i=1}^{r}H_{2k_i}(\la_i)\oplus\bigoplus_{j=1}^s \beta_j\Gamma_{2m_j}\oplus\bigoplus_{k=1}^t\gamma_k \Gamma_{2n_{k}+1}\in
\cbunh\left(\bigoplus_{i=1}^{\kappa+\sigma+\nu} H_2(\mu_i)\oplus\bigoplus_{j=1}^{t}\alpha_j \right),
$$
where $\mu_i$ and $\alpha_j$ are as in the statement. Since $n-2(\kappa+\sigma+\nu)=t$, Lemma \ref{congruencebundle_lem} and the equation above imply that $A \in \overline{\mathcal{G}_\ell}$ for $\ell = \kappa+\sigma+\nu$.

If $A$ is not invertible, there is a sequence of invertible matrices, ${\cal S}:=\{A_k\}_{k\in\mathbb N}$, which converges to $A$. Since the number of bundle closures in the right-hand side of the equality in (b1) is finite, there will be a subsequence of $\cal S$ converging to $A$ whose terms (matrices) belong to the same bundle closure. Therefore, $A$ itself belongs to this bundle closure as well. This completes the proof of (b1).

Next, we prove (b2). For this purpose, we are going to prove that if there exists a matrix $\displaystyle A \in \overline{\mathcal{G}_{\ell_1}} \cap \mathcal{G}_{\ell_2}$, then $\ell_1=\ell_2$. Note first that $\displaystyle A \in \overline{\mathcal{G}_{\ell_1}}$ implies that $A+\lambda A^*$ has at least $n-2\ell_1$ unit eigenvalues (counting multiplicities) by Lemma \ref{mod1evals_lem}, while $\displaystyle A \in \mathcal{G}_{\ell_2}$ implies that $A + \lambda A^*$ is regular and has exactly $n-2\ell_2$ unit eigenvalues, by \cite[Th. 2.13-(b)]{dd25}. Thus $\ell_1 \geq \ell_2$.

Since $A \in \overline{\mathcal{G}_{\ell_1}} = \cbunh\left(\bigoplus_{i=1}^{\ell_1} H_2(\mu_i)\oplus\bigoplus_{j=1}^{n-2\ell_1}\alpha_j\right)$, there is a sequence of invertible matrices $\{A_k\}_{k\in\mathbb N}$ such that
  \begin{itemize}
      \item[(i)] $A_k$ has $^*$CFC equal to $\left(\bigoplus_{i=1}^{\ell_1} H_2(\mu_{i,k})\oplus\bigoplus_{j=1}^{n-2\ell_1}\alpha_{j,k}\right)$, with $|\mu_{i,k}|>1$ and $\mu_{i,k}\neq\mu_{i',k}$, for $i\neq i'$, and $|\alpha_{j,k}|=1$ with $\alpha_{j,k}^2\neq \alpha_{j',k}^2$ for $j\neq j'$.

      \item[(ii)] $\{A_k\}_{k\in\mathbb N}$ converges to $A$.
  \end{itemize}
As a consequence of (ii), the sequence of regular matrix pencils $\{A_k+\la A_k^*\}_{k\in\mathbb N}$ converges to the regular pencil $A+\la A^*$.

Using again \cite[Th. 2.13-(b)]{dd25}, we get that the eigenvalues of the pencil $A_k+\la A_k^*$ are $-\mu_{i,k}, -1/\overline\mu_{i,k}$ and $-\alpha_{j,k}^2$, for $i=1,\hdots,\ell_1$ and $j=1,\hdots, n-2\ell_1$, i.e., $A_k+\la A_k^*$ has $2 \ell_1$ eigenvalues of modulus different from $1$ and $n-2\ell_1$ eigenvalues of modulus $1$. Similarly, since $A \in \mathcal{G}_{\ell_2}$, its $^*$CFC is $\left(\bigoplus_{i=1}^{\ell_2} H_2(\widetilde \mu_i)\oplus\bigoplus_{j=1}^{n-2\ell_2} \widetilde \alpha_j\right)$, with the restrictions in the parameters as in \eqref{def.gen*bund}, so $A+\la A^*$ has $2 \ell_2$ eigenvalues $-\widetilde \mu_{i}, -1/\overline{\widetilde \mu_{i}}$ of modulus different from $1$ and $n-2\ell_2$ eigenvalues $-\widetilde{\alpha}_j^2$ of modulus $1$. All these $n$ eigenvalues are distinct.
If $\ell_1>\ell_2$ (namely, $n-2\ell_2>n-2\ell_1$), by the continuity of eigenvalues of regular matrix pencils (see, for instance, Th. 2.1 in \cite[Ch. VI]{stewartsun}), there must be some sequence $\{-\mu_{i,k}\}_{k\in\mathbb N}$ or $\{-1/\overline\mu_{i,k}\}_{k\in\mathbb N}$ which converges to some $-\widetilde{\alpha}_j^2$. Assume, without loss of generality, that $\{-\mu_{1,k}\}_{k\in\mathbb N}$ converges to $-\widetilde{\alpha}_1^2$. Then $\{-1/\overline\mu_{1,k}\}_{k\in\mathbb N}$ converges to $-1/\overline{\widetilde{\alpha}_1^2} =-\widetilde{\alpha}_1^2$, since $|\widetilde{\alpha}_1^2|=1$. As a consequence, the eigenvalue $-\widetilde{\alpha}_1^2$ appears twice in $A+\la A^*$, which is a contradiction because the eigenvalues of this pencil are all distinct. Therefore, $\ell_1= \ell_2$, as wanted.

\medskip

Finally, we prove (a). Using again the fact that any non-invertible matrix is the limit of a sequence of invertible matrices, it is enough to prove that any invertible matrix $A\in \mathbb{C}^{n\times n}$ is in the closure of the bundle defined in \eqref{mainidentityc}. The proof has two parts. In the first one, we proceed as in the proof of (b1), replacing $^*$congruence by congruence. Thus, we start by assuming that $A$ is invertible. Then, by Theorem \ref{hs_th}-(a), $A$ is congruent to a direct sum of the form  $\bigoplus_{i=1}^{r}H_{2k_i}(\la_i)\oplus\bigoplus_{j=1}^s \Gamma_{2m_j}\oplus\bigoplus_{k=1}^t \Gamma_{2n_{k}+1}$, for some $\la_i \in\C$ such that $0\ne \lambda_i \ne (-1)^{k_i + 1}$, for all $i=1,\hdots,r$, and with $2(k_1+\cdots+k_r)+2(m_1+\cdots+m_s)+2(n_1+\cdots+n_t)+t=n$. Let us set, as before, $\kappa:=k_1+\cdots+k_r$, $\sigma:=m_1+\cdots+m_s,$ and $\nu:=n_1+\cdots+n_t$, so that $2(\kappa+\sigma+\nu)+t=n$. Then, using Lemmas \ref{decomp_lem}, \ref{lemm.dirsumbund4} and \ref{congruencebundle_lem}, and mimicking the arguments in the proof of (b1), we end up with
\begin{equation} \label{eq.1(a)}
A\in\cbunc\left(\bigoplus_{i=1}^{\ell} H_2(\mu_i)\oplus I_{n-2\ell}\right),
\end{equation}
for $\ell = \kappa + \sigma + \nu$, and where $|\mu_i| >1$ or $\mu_i = e^{i\theta}$, with $0<\theta < \pi$, and $\mu_i\neq\mu_{i'}$, if $i\neq i'$, for all $i,i'=1,\hdots,\ell$.

In the second part of the proof of (a), we will prove that
\begin{equation} \label{eq.2(a)even}
\bunc\left(\bigoplus_{i=1}^{\ell} H_2(\mu_i)\oplus I_{n-2\ell}\right)   \subseteq
\cbunc\left(\bigoplus_{i=1}^{n/2} H_2(\mu_i) \right), \quad \mbox{if $n$ is even},
\end{equation}
and that
\begin{equation} \label{eq.2(a)odd}
\bunc\left(\bigoplus_{i=1}^{\ell} H_2(\mu_i)\oplus I_{n-2\ell}\right)   \subseteq
\cbunc\left(\bigoplus_{i=1}^{\lfloor n/2 \rfloor} H_2(\mu_i) \oplus 1\right), \quad \mbox{if $n$ is odd},
\end{equation}
where in both cases the parameters $\mu_i$ satisfy the conditions in the statement. These inclusions combined with \eqref{eq.1(a)} imply that any invertible $n\times n$ matrix is in the closure of the bundle defined in \eqref{mainidentityc}. To prove \eqref{eq.2(a)even} and \eqref{eq.2(a)odd}, observe that
\begin{align*}
I_{n-2\ell} \quad \mbox{is congruent to} \quad S:= \overbrace{\begin{bmatrix}
        0&1\\1&0
\end{bmatrix}\oplus\cdots\oplus\begin{bmatrix}
        0&1\\1&0
\end{bmatrix}}^{n/2-\ell} ,   \; \quad \mbox{if $n$ is even}, \\
I_{n-2\ell} \quad \mbox{is congruent to} \quad
S:= \overbrace{\begin{bmatrix}
        0&1\\1&0
\end{bmatrix}\oplus\cdots\oplus\begin{bmatrix}
        0&1\\1&0
\end{bmatrix}}^{\lfloor n/2 \rfloor -\ell}\oplus 1,  \; \quad \mbox{if $n$ is odd}.
\end{align*}
This follows from \cite[Lemma 2.1]{hs06} because the cosquare of $S$, that is, $S^{-\top} S$, is  $S^{-\top} S = I_{n-2\ell}$. Therefore, if $C\in \bunc\left(\bigoplus_{i=1}^{\ell} H_2(\mu_i)\oplus I_{n-2\ell}\right)$, then there is an invertible $P \in \mathbb{C}^{n\times n}$ such that
$$
C = \left\{
\begin{array}{l} P \left(
 \bigoplus_{i=1}^{\ell} H_2(\widetilde \mu_i) \oplus    \overbrace{\begin{bmatrix}
        0&1\\1&0
\end{bmatrix}\oplus\cdots\oplus\begin{bmatrix}
        0&1\\1&0
\end{bmatrix}}^{n/2-\ell} \right) P^\top ,   \; \quad \mbox{if $n$ is even},\\[10mm]
P \left( \bigoplus_{i=1}^{\ell} H_2(\widetilde \mu_i)   \oplus
\overbrace{\begin{bmatrix}
        0&1\\1&0
\end{bmatrix}\oplus\cdots\oplus\begin{bmatrix}
        0&1\\1&0
\end{bmatrix}}^{\lfloor n/2 \rfloor -\ell} \oplus 1 \right) P^\top,  \; \quad \mbox{if $n$ is odd},
\end{array}
\right.
$$
where the parameters $\widetilde \mu_i$ are all distinct and satisfy the conditions in the statement. Observe that $C = \lim_{k\rightarrow \infty} C_k$, where the sequence of matrices $\{C_k\}_{k \in \mathbb{N}}$ is defined as follows
$$
C_k := \left\{
\begin{array}{l} P \left(
 \bigoplus_{i=1}^{\ell} H_2(\widetilde \mu_i) \oplus    \overbrace{\begin{bmatrix}
        0&1\\1+ \frac{1}{k}&0
\end{bmatrix}\oplus\cdots\oplus\begin{bmatrix}
        0&1\\1+\frac{1}{(n/2-\ell)k}&0
\end{bmatrix}}^{n/2-\ell} \right) P^\top ,   \; \quad \mbox{if $n$ is even},\\[10mm]
P \left( \bigoplus_{i=1}^{\ell} H_2(\widetilde \mu_i)   \oplus
\overbrace{\begin{bmatrix}
        0&1\\1+ \frac{1}{k} &0
\end{bmatrix}\oplus\cdots\oplus\begin{bmatrix}
        0&1\\1+\frac{1}{(\lfloor n/2 \rfloor -\ell)k}&0
\end{bmatrix}}^{\lfloor n/2 \rfloor -\ell}\oplus 1 \right) P^\top,  \; \quad \mbox{if $n$ is odd}.
\end{array}
\right.
$$
For all $k\geq k_0$, where $k_0$ is sufficiently large, $C_k$ belongs to $\bunc\left(\bigoplus_{i=1}^{n/2} H_2(\mu_i) \right)$, if $n$ is even, or to $\bunc\left(\bigoplus_{i=1}^{\lfloor n/2 \rfloor} H_2(\mu_i) \oplus 1\right)$, if $n$ is odd. Hence, $C$ belongs to $\cbunc\left(\bigoplus_{i=1}^{n/2} H_2(\mu_i) \right)$, if $n$ is even, or to $\cbunc\left(\bigoplus_{i=1}^{\lfloor n/2 \rfloor} H_2(\mu_i) \oplus 1\right)$, if $n$ is odd, which proves \eqref{eq.2(a)even} and  \eqref{eq.2(a)odd}.
\end{proof}

\begin{remark} \label{rem.codimension} {\rm (Codimensions of the generic bundles) Using Theorem 2 and the formula (56) in \cite{dd10}, it can be seen that the congruence bundle in \eqref{mainidentityc} is the unique congruence bundle with codimension $0$. Analogously, using \cite[Theorem 3.3]{dd11} and the formula for the real codimension of $^*$congruence bundles at the bottom of \cite[p. 462]{dd11}, it can seen that the $\lfloor n/2 \rfloor + 1$ bundles defined in \eqref{def.gen*bund} are the unique $^*$congruence bundles with real codimension zero.
}
\end{remark}

\section{Openness of the generic bundles in $\C^{n\times n}$ }\label{sec.openness}
As announced in the introduction, this section is devoted to prove that the bundles introduced in Theorem \ref{main_th} are open subsets of $\mathbb{C}^{n \times n}$. Thus, combining Theorems \ref{main_th} and \ref{th.opennness}, we get that the unique congruence bundle introduced in \eqref{mainidentityc} (with slightly different expressions for $n$ even or odd) is an open and dense subset of $\mathbb{C}^{n \times n}$ and, so, it can be properly termed as a {\em generic} subset of $\mathbb{C}^{n \times n}$, according to the standard definition in topology. Analogously, from Theorems \ref{main_th} and \ref{th.opennness}, we get that the union, $\bigcup_{\ell=0}^{\lfloor n/2\rfloor} \mathcal{G}_\ell$, of the $^*$congruence bundles defined in \eqref{def.gen*bund} is an open and dense subset of $\mathbb{C}^{n \times n}$, and, so, generic. In Theorem \ref{th.opennness}, we will use the Frobenius norm of a matrix \cite[p. 321]{hj13}, denoted by $\|A\|_F$, to be specific, but any other norm can be used.

\begin{theorem} \label{th.opennness}
\begin{itemize}
    \item[\rm(a)] The set $\mathcal{G}$ defined in \eqref{mainidentityc} is an open subset of $\mathbb{C}^{n \times n}$. That is, for any $A\in \mathcal{G}$, there exists a number $\varepsilon >0$ such that
   $$
   \mathcal{U}_\varepsilon (A) := \{ S \in \mathbb{C}^{n \times n} \, : \,  \|A-S\|_F < \varepsilon\} \subset \mathcal{G} .
   $$

    \item[\rm(b)]  The subset  $\mathcal{G}_\ell$ defined in \eqref{def.gen*bund} is an open subset of $\mathbb{C}^{n \times n}$ for any $\ell = 0,1, \ldots , \lfloor n/2 \rfloor$. That is, for any $A\in \mathcal{G}_\ell$, there exists a number $\varepsilon >0$ such that
   $$
   \mathcal{U}_\varepsilon (A) := \{ S \in \mathbb{C}^{n \times n} \, : \,  \|A-S\|_F < \varepsilon\} \subset \mathcal{G}_\ell .
   $$

\end{itemize}
\end{theorem}
\begin{proof}
We prove first in detail part (b), which is longer. Then, we sketch the proof of part (a).

{\em Proof of} (b). The proof has two parts. In the first part, we prove that if $A\in \mathcal{G}_\ell$ and $\varepsilon >0$ is sufficiently small, then
$\mathcal{U}_\varepsilon (A) \subset \bigcup_{k=0}^{\lfloor n/2\rfloor} \mathcal{G}_k$. With this result at hand, in the second part we prove that $\mathcal{U}_\varepsilon (A) \subset \mathcal{G}_\ell$ for some $\varepsilon >0$. If $A \in \mathcal{G}_\ell$, then $A$ is invertible, the pencil $A + \lambda A^*$ is regular and, from \cite[Theorem 2.13-(b)]{dd25}, this pencil has $n$ distinct eigenvalues, all different from zero and infinite, such that $\ell$ of them have modulus larger than $1$, $\ell$ of them have modulus smaller than $1$, and $n-2 \ell$ have modulus equal to $1$. Thus, for any $\varepsilon >0$ sufficiently small, if $S \in \mathcal{U}_\varepsilon (A)$, then $S$ is invertible, the pencil $S + \lambda S^*$ is regular, and it has $n$ distinct eigenvalues, all different from zero and infinite, by the continuity of the eigenvalues of regular pencils (Th. 2.1 in \cite[Ch. VI]{stewartsun}). Combining this with the relationship between the KCF of $S + \lambda S^*$ and the $^*$CFC of $S$ \cite[Theorem 2.13-(b)]{dd25}, one gets that $S \in \mathcal{G}_k$, for some $k = 0,1, \ldots , \lfloor n/2 \rfloor$. So, $\mathcal{U}_\varepsilon (A) \subset \bigcup_{k=0}^{\lfloor n/2\rfloor} \mathcal{G}_k$ for any $\varepsilon >0$ sufficiently small.

For the second part of the proof, that is, to prove that there exists an $\varepsilon >0$ such that  $\mathcal{U}_\varepsilon (A) \subset \mathcal{G}_\ell$, we proceed by contradiction and assume that $\mathcal{U}_\varepsilon (A) \not\subset \mathcal{G}_\ell$ for any $\varepsilon >0$. This and the result above imply that  $\mathcal{U}_\varepsilon (A) \cap \bigcup_{k=0, k\ne \ell}^{\lfloor n/2\rfloor} \mathcal{G}_k \ne \emptyset$ for any $\varepsilon >0$ sufficiently small. Therefore, $A \in \overline{\bigcup_{k=0, k\ne \ell}^{\lfloor n/2\rfloor} \mathcal{G}_k} = \bigcup_{k=0, k\ne \ell}^{\lfloor n/2\rfloor} \overline{\mathcal{G}_k}$, which is in contradiction with Theorem \ref{main_th}-(b2) because $A \in \mathcal{G}_\ell$.

\medskip

{\em Proof of} (a). It is completely analogous to the first part of the proof of (b), with the only difference that \cite[Th. 4]{d16} has to be used to relate the CFC of $A$ with the KCF of the pencil $A + \lambda A^\top$. Note that a second part of the proof is not needed in this case because there is only one generic congruence bundle $\mathcal{G}$. We omit the details for brevity.
\end{proof}

\section{Numerical experiments} \label{sec.numerexp}
We have used several times in this paper the fact that, for any matrix $A\in \mathbb{C}^{n\times n}$, there is a bijection between the CFC of $A$ and the KCF of the $\top$-palindromic pencil $A + \lambda A^\top$, which can be found in \cite[Th. 4]{d16}, and a very close connection between the $^*$CFC of $A$ and the KCF of the $*$-palindromic pencil $A + \lambda A^*$, which can be found in \cite[Th. 2.13-(b)]{dd25}. In particular, this connection implies that if $A$ belongs to some of the generic $^*$congruence bundles defined in \eqref{def.gen*bund}, i.e., $A\in \mathcal{G}_\ell$ for some $\ell = 0, 1, \ldots , \lfloor n/2 \rfloor$, then $A + \lambda A^*$ has exactly $n-2\ell$ eigenvalues of modulus $1$. This fact may seem surprising at  first glance, because the value $1$ of the modulus seems very particular. At first glance, one might think that eigenvalues of modulus $1$ should not appear very often in $*$-palindromic pencils or, if they do, they should appear in a very small number. The main purpose of this section is to show through some numerical experiments that eigenvalues of modulus $1$ arise very often in random $*$-palindromic pencils and, in fact, in any of the numbers predicted by Theorem \ref{main_th}-(b). In the last part of the section, we briefly discuss numerical experiments with random $\top$-palindromic pencils, which show that no eigenvalues of modulus $1$ are encountered for even size, and only one (the eigenvalue $-1$) for odd size. This is also in agreement with Theorem \ref{main_th}-(a).

In the first numerical experiment, we have obtained the number of unit eigenvalues in $m$ random $*$-palindromic pencils with size $n\times n$ with the following {\sc matlab} code:

\begin{verbatim}
function norm1(n,m)
% counts the number of eigenvalues of modulus 1 in m random *-palindromic pencils
% n is the size of the pencil
% m is the number of pencils
counter=zeros(m,1);
for i=1:m
    a=rand(n)+sqrt(-1)*rand(n);
    e=eig(a,-a');
    normi=abs(e);
    for j=1:n
        if abs(normi(j)-1)/norm(a)<=10^(-14)
            counter(i)=counter(i)+1;
        else
            counter(i)=counter(i);
        end
    end
end
plot(counter,'o')

x = unique(counter);
N = numel(x);
count = zeros(N,1);
for k = 1:N
   count(k) = sum(counter==x(k));
end
\end{verbatim}

The results obtained for $n=24$ and $25,$ and $m=10^4$ are displayed in Table \ref{rand(n)_table}. They are partially in accordance with Theorem \ref{main_th}-(b). In particular, the number of unit eigenvalues is odd when the size $n$ is odd, and it is even when $n$ is even. However, not all possible numbers of unit eigenvalues predicted by Theorem \ref{main_th}-(b) show up. In particular, for $n=25$ only pencils with $1,3,5,7,9,$ and $11$ unit eigenvalues arise (namely pencils with $13,15,17,19,21,23,$ and $25$ unit eigenvalues are missing). Something similar happens with $n=24$, for which only pencils with $0,2,4,6,8,$ and $10$ unit eigenvalues show up. The obtained results are rather similar if {\tt randn(n)} is used instead of {\tt rand(n)}. Informally speaking, this indicates that the generic $^*$congruence bundles in \eqref{def.gen*bund} giving rise to higher number of unit eigenvalues are not located in the space $\mathbb{C}^{n\times n}$ in the same ``probable'' regions of the random $*$-palindromic pencils that have been generated. Therefore, we replace {\tt a} in the first line of the outer ``for" loop of the {\tt function norm1(n,m)} by
\begin{equation}\label{matrixa}
{\tt a=randi(m,n)+i*log(i)*eye(n)/5+sqrt(-1)*(randi(m,n)+i*log(i)*eye(n)/5)},
\end{equation}
for {\tt i=1:m}, where {\tt m} is the number of pencils in the experiment. In this case, we get the results displayed in Table \ref{25-10000_table}, which are in accordance with Theorem \ref{main_th}-(b). Observe that all possible generic bundles appear as an output. As {\tt i} increases, the number of unit eigenvalues seems to increase, as can be seen in Figure \ref{plot.figure}. This resembles very much the results obtained in \cite[\S7]{ddd24} for the number of real eigenvalues of Hermitian pencils. Taking into account that the eigenvalues of matrix pencils are generically simple, we may infer that the pencil {\tt (a,-a')} has $k$ unit eigenvalues, so the KCF of {\tt (a,-a')} is $\bigoplus_{i=1}^k(\la - \alpha_i) \oplus \bigoplus_{i=1}^{(n-k)/2}\begin{bsmallmatrix}
    \la-\mu_i & 0 \\ 0 & \la-(\overline\mu_i)^{-1} \end{bsmallmatrix}$, for some $\mu_i\in\C$ with $|\mu_i|>1$ and $\mu_i\neq\mu_{i'}$,  $\beta_i\in\C$ with $|\beta_i| = 1$ and $\beta_i\neq\beta_{i'}$ for $i\neq i'$, where $\beta_i$ are the unit eigenvalues of {\tt (a,-a')}. Then, the $^*$CFC of the matrix {\tt a} is the one in \eqref{def.gen*bund} for $n-2\ell=k$ and with $\alpha_i = \pm \sqrt{\beta_i}$ (see \cite[Th. 2.13-(b)]{dd25}).

\begin{table}[h]
    \centering
    \begin{tabular}{|c|c||c|c|}\hline
    \multicolumn{2}{|c|}{$n=25,\ m=10^4$}&\multicolumn{2}{|c|}{$n=24,\ m=10^4$}\\\hline\hline
         Number of e-vals&Frequency &Number of e-vals&Frequency \\\hline
         1&780&0&92\\
         3&3558&2&2067\\
         5&3954&4&4459\\
         7&1483&6&2785\\
         9&215&8&552\\
         11&10&10&45\\\hline
    \end{tabular}
    \caption{Number of unit eigenvalues in $m$ random $n\times n$ $*$-palindromic pencils of the type {\tt a=rand(n)+sqrt(-1)*rand(n)}, with their frequency.}
    \label{rand(n)_table}
\end{table}

\begin{table}[]
    \centering
    \begin{tabular}{|c|c||c|c|}\hline
    \multicolumn{2}{|c|}{$n=25,\ m=10^4$}&\multicolumn{2}{|c|}{$n=24,\ m=10^4$}\\\hline\hline
         Number of e-vals&Frequency &Number of e-vals&Frequency \\\hline
         1&37&0&6\\
         3&395&2&162\\
         5&1031&4&806\\
         7&1086&6&1146\\
         9&886&8&965\\
         11&805&10&828\\13&819&12&865\\15&799&14&807\\17&810&16&834\\19&863&18&878\\21&968&20&983\\23&1087&22&1099\\25&414&24&621\\\hline
    \end{tabular}
    \caption{Number of unit eigenvalues in $m$ $*$-palindromic $n\times n$ pencils $A+\la A^*$ with $A$ as in \eqref{matrixa}, and their frequency.}
    \label{25-10000_table}
\end{table}

\begin{figure}[H]

\centering
\subfloat[Size $25\times 25$]{\label{fig:capparatus}
\centering
\includegraphics[width=0.6\linewidth]{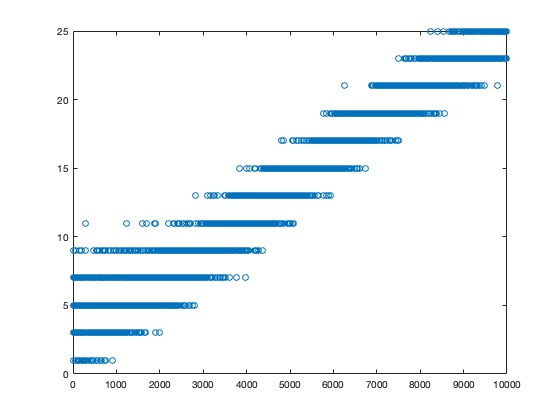}
}
\hfill
\subfloat[Size $24\times24$]{\label{fig:cdiagram}
\centering

\includegraphics[width=0.6\linewidth]{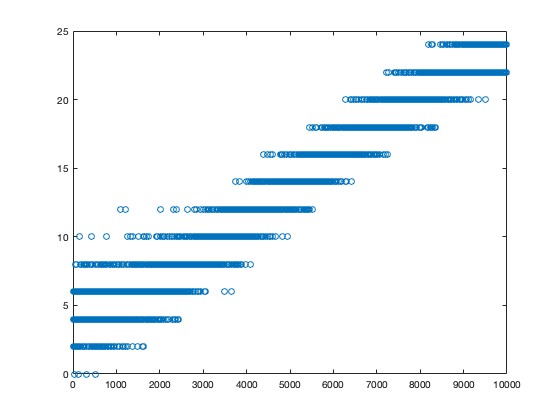}
}

\caption{Number of unit eigenvalues of pencils $A+\la A^*$, with $A$ as in \eqref{matrixa}, for {\tt n=25} (top) and {\tt n=24} (bottom), {\tt m=10000}, and {\tt i=1:10000} (horizontal axis).}
\label{plot.figure}
\end{figure}

Finally, if we run the code {\tt norm1(n,m)} replacing {\tt e=eig(a,-a')} by {\tt e=eig(a,-a.')}, so that the pencil {\tt (a,-a.')} is $\top$-palindromic, then, for {\tt n=24} and {\tt m=10000} we get that all $10000$ pencils have no unit eigenvalues, whereas for {\tt n=25} all $10000$ pencils have just one unit eigenvalue, as expected.

\bigskip

\noindent{\bf Acknowledgments}. This work has been partially supported by grants PID2023-147366NB-I00 funded by MICIU/AEI/10.13039/501100011033 and FEDER/UE, and RED2022-134176-T.

\bibliographystyle{plain}

\end{document}